\documentclass[a4paper,leqno,12pt,twoside]{amsart}
\usepackage{amsmath,amsfonts,amssymb,amsthm,amscd}
\usepackage[utf8]{inputenc}
\usepackage[T1]{fontenc}
\usepackage[english]{babel}
\usepackage[colorlinks=true,citecolor=blue, urlcolor=blue, linkcolor=blue,pagebackref]{hyperref}
\usepackage[top=1.2in,bottom=1.2in,left=1.in,right=1.in]{geometry} 
\usepackage{graphicx}
\usepackage{paralist}
\usepackage{tabto}
\usepackage{standalone}
\usepackage{tikz}
\usetikzlibrary{matrix}
\usetikzlibrary{arrows}
\usepackage{xfrac}
\usepackage{amsthm, amsmath, amssymb,latexsym}
\usepackage{tikz-cd}
\usepackage[alphabetic,backrefs,lite]{amsrefs}
\usepackage[all]{xy} 

\usepackage{enumerate}
\usepackage{xcolor}
\usepackage{aliascnt}

\newtheorem{theorem}{Theorem}[section]
\newtheorem{lemma}[theorem]{Lemma}
\newtheorem{corollary}[theorem]{Corollary}
\newtheorem{proposition}[theorem]{Proposition}
\newtheorem{remark}[theorem]{Remark}
\newtheorem{fact}[theorem]{Fact}
\newtheorem{definition}[theorem]{Definition}

\newtheorem{question}[theorem]{Question}
 
\newtheorem*{theorem*}{Theorem}
\newtheorem*{THA}{Theorem A}
\newtheorem*{THB}{Theorem B}
\newtheorem*{THC}{Theorem C}
\newtheorem*{THD}{Theorem D}

\newcommand{\nc}{\newcommand} 
\nc{\cH}{{\mathcal H}}
\nc{\cA}{{\mathcal A}}
\nc{\cG}{{\mathcal G}}
\nc{\cC}{{\mathcal C}}
\nc{\cD}{{\mathcal D}}
\nc{\cO}{{\mathcal O}}
\nc{\cI}{{\mathcal I}}
\nc{\cB}{{\mathcal B}}
\nc{\cY}{{\mathcal Y}}
\nc{\cK}{{\mathcal K}} 
\nc{\cX}{{\mathcal X}}
\nc{\cS}{{\mathcal S}}
\nc{\cE}{{\mathcal E}}
\nc{\cF}{{\mathcal F}}
\nc{\cZ}{{\mathcal Z}}
\nc{\cQ}{{\mathcal Q}}
\nc{\cN}{{\mathcal N}}
\nc{\cP}{{\mathcal P}}
\nc{\cL}{{\mathcal L}}
\nc{\cM}{{\mathcal M}}
\nc{\cT}{{\mathcal T}}
\nc{\cW}{{\mathcal W}}
\nc{\cU}{{\mathcal U}}
\nc{\cJ}{{\mathcal J}}
\nc{\cV}{{\mathcal V}}
\nc{\bH}{{\mathbb H}}
\nc{\bA}{{\mathbb A}}
\nc{\bG}{{\mathbb G}}
\nc{\bC}{{\mathbb C}}
\nc{\bO}{{\mathbb O}}
\nc{\bI}{{\mathbb I}}
\nc{\bB}{{\mathbb B}}
\nc{\bY}{{\mathbb Y}}
\nc{\bK}{{\mathbb K}} 
\nc{\bX}{{\mathbb X}}
\nc{\bS}{{\mathbb S}}
\nc{\bE}{{\mathbb E}}
\nc{\bF}{{\mathbb F}}
\nc{\bZ}{{\mathbb Z}}
\nc{\bQ}{{\mathbb Q}}
\nc{\bN}{{\mathbb N}}
\nc{\bP}{{\mathbb P}}
\nc{\bL}{{\mathbb L}}
\nc{\bM}{{\mathbb M}}
\nc{\bT}{{\mathbb T}}
\nc{\bW}{{\mathbb W}}
\nc{\bU}{{\mathbb U}}
\nc{\bD}{{\mathbb D}}
\nc{\bJ}{{\mathbb J}}
\nc{\bV}{{\mathbb V}}
\nc{\bbZ}{{\mathbb Z}}
\nc{\bR}{{\mathbb R}}
\nc{\fr}{{\rightarrow}}
\nc{\co}{{\nabla}}
\nc{\Pic}{{\mbox{Pic}}}

\nc{\cu}{{\barline{\nabla}}}

\title{On the maximal variation problem and Lefschetz pencils}

 \author{D. Bricalli, G.P. Pirola}
 
\address{Davide Bricalli  \\ Universit\`a degli Studi di Pavia  \\ Dipartimento di Matematica \\ Via Ferrata 1  \\ 27100 Pavia, Italy}
\email{davide.bricalli@unipv.it \ bricalli@altamatematica.it}

 \address{Gian Pietro Pirola  \\ Universit\`a degli Studi di Pavia  \\ Dipartimento di Matematica \\ Via Ferrata 1  \\ 27100 Pavia, Italy  }
 \email{gianpietro.pirola@unipv.it}

\begin{document}
\thanks{
\textit{Keywords}: Maximal variation, Hodge theory, Lefschetz pencils, Lefschetz properties, hypersurfaces, Enriques surfaces.}

\subjclass[2020]{Primary: 14J10; Secondary: 14D05, 14H10, 13E10.}

\begin{abstract}
We study the maximal variation problem for linear systems associated with a very ample line bundle, using Hodge theory and Picard–Lefschetz theory. We provide an affirmative answer to the maximal variation problem for a broad class of smooth projective varieties. This includes varieties $X$ of dimension $n\geq2$ with $p_g=h^{n,0}(X)>0$ and $H^{n-1,0}(X)=\{0\}$, Enriques surfaces, irregular surfaces with maximal Albanese dimension, smooth hyperk\"ahler varieties, and all the smooth not Fano hypersurfaces in $\bP^n$. As a consequence, by a result of Beauville, we establish a Lefschetz property for the Jacobian rings of smooth hypersurfaces in $\bP^n$ of degree $n+1$.

\end{abstract}

\maketitle

\section*{Introduction}

Let $\pi:\cZ\rightarrow U$ be a smooth family of varieties and let $Y=\pi^{-1}(u)$, for $u\in U$, be its general member. Following \cite{DH}, we say that $\pi$ has maximal variation if only finitely many elements of the family are isomorphic to $Y$. In other words, whenever a moduli space $\cM$ of elements of the family exists, the map $U\rightarrow \cM$ is generically finite. Here, we consider families of smooth hyperplane sections of a smooth projective complex variety $X$. Given a very ample line bundle $L$ on a smooth projective variety $X$, we say that the linear system $|L|$ has maximal variation if the corresponding family of smooth hyperplane sections does. This happens, for instance, if there doesn't exist a variety parametrizing an isotrivial family of hyperplane sections of $X$ and passing through the general point of $\bP(H^0(X,L))$. Moreover, we say that the variety $X$ itself has maximal variation if every very ample line bundle on $X$ does (see Definition \ref{DEF:maximalvariationofL}). \\
In particular, understanding how hyperplane sections vary in moduli is a classical problem in algebraic geometry, closely related to deformation theory and Hodge theory. As recalled by Beauville in \cite{Beau2}, there is an apparently simple differential criterion which ensures the maximal variation property for a linear system $|L|$: if the tangent bundle of $X$ restricted to a general element $Y$ of $|L|$ has no sections, i.e. $H^0(Y,T_{X|Y})=0$, then $|L|$ has maximal variation. This is the classical infinitesimal approach: in this direction many results have been proved dealing for example with the case of K3 surfaces (see \cite{DH} and \cite{Bak}) or in the case of abelian varieties (see \cite{Beau2}). In our opinion, this turns out to be more delicate and more challenging to prove than it might initially appear: the maximal variation problem for linear systems is still open for many examples, despite the expectation that stronger results should hold. For instance, in \cite{Sch} Chad Schoen, using Hodge theory, proved that if $X$ is a very general hypersurface in $\bP^n$, with $n\geq3$, of degree $d\geq n+2$, then it can't be covered by a product of varieties. So, in this case, the corresponding map on the moduli of hyperplane sections is actually finite, and not just generically finite.
\medskip

We tackle this problem from a Hodge theoretic point of view. With the use of Picard-Lefschetz theory, we consider the case of varieties with positive geometric genus and we prove the following:
\begin{THA}[Theorem \ref{THM:main}]
    Let $X$ be a smooth complex projective variety of dimension $n\geq2$, embedded in the projective space $\bP(V^*)$, where $V=H^0(X,L)$ is the space of global sections of a very ample line bundle $L$ over $X$. Assume that:
    \begin{enumerate}
        \item $p_g=\dim(H^{n,0}(X))>0$;
        \item $H^{n-1,0}(X)=0$.
    \end{enumerate}
    Then there is no isotrivial subvariety passing through the general point of $\bP(V)$. In particular, $X$ has maximal variation.
\end{THA}

By an {\it{isotrivial subvariety}} we mean (see Definition \ref{DEF:isotrivial}) a connected subvariety $Z$ of positive dimension in $\bP(V)$ such that the induced family of hyperplane sections of $X$ is isotrivial on an open subset of $Z$. Under these assumptions, we show, in particular, that there is no positive-dimensional direction in $\bP(V)$ along which the isomorphism class of the corresponding hyperplane sections remains constant on a Zariski open subset.\\ 
Roughly speaking, the main point is the following: given a family of rational dominant maps $C\times Y\dashrightarrow X$, where $Y$ is the general hyperplane section of $X$ and $C$ is a curve, we look at the pull-back of top-forms from $X$. This gives some invariant part of the cohomology of the product $C\times Y$ and such a behavior cannot interact with the vanishing Hodge structure of $X$. In particular, this is not possible if the purely Hodge theoretic assumptions in the statement of Theorem A hold. More precisely, we assume, by contradiction the existence of an isotrivial family passing through the general point of $\bP(V)$, where $V$ is the space of global sections $H^0(X,L)$ of a very ample line bundle $L$ on $X$. After a base change, we can then consider the dominant and generically finite rational map $\psi:C\times Y\dashrightarrow X$. We focus then on the induced injective pull-back $\psi^*$ at the level of top forms. After applying the Künneth decomposition to $H^n(C\times Y)$, by the use of the Hodge theoretical assumptions, we are able to identify the cohomological summand in which the image $\psi^*(H^{n,0}(X))$ must lie, namely the one defined by the vanishing cohomology of the general hyperplane section. The monodromy and the Picard-Lefschetz formula come then into the picture, allowing us to prove that the map $\psi^*$ has to be the zero map, yielding a contradiction. In the same way, one can actually prove that in the assumptions of the above Theorem neither a birationally isotrivial subvariety (see Definition \ref{DEF:isotrivial}) passing through the general point of $\bP(V)$ can exist: in other words if $Y$ is a general hyperplane section of $X$, for a variety $X$ as in the Theorem A, then there are at most finitely many sections which are birational to $Y$.\\
For example, this conclusion holds for smooth hyperk\"ahler varieties and for a smooth complex projective surface $S$ with $p_g(S)>0$ and $q(S)=0$.
\smallskip

From our techniques, we then recover the result on the K3 surfaces in the case of a very ample line bundle (see \cite{DH}, \cite{Bak}, \cite{ABS}, \cite{CG}, \cite{CD}, \cite{CDS},...) but at first sight the case of Enriques surfaces (studied for example in \cite{CDGK}) seems to be out of our control. However, by lifting the problem on the corresponding K3 surface and exploiting the properties of the vanishing cohomology arising from the associated étale $2:1$ cover (studied and described in \cite{LP}), we prove the prove the maximal variation property also for smooth Enriques surfaces. With completely different techniques we deal also with the case of smooth projective irregular surfaces with maximal Albanese dimension:

\begin{THB}[Theorem \ref{THM:Enriques} and Corollary \ref{COR:maxvarirrsurf}]
   Smooth Enriques surfaces and smooth projective surfaces with maximal Albanese dimension have maximal variation.    
\end{THB}

One can actually observe that with the same approach, the result holds also for smooth projective vareities (of higher dimension) with maximal Albanese dimension. On the other hand, the case of surfaces with Albanese dimension $1$ can't be, at first sight, naturally deduced from our techniques. Actually, in this case there are some examples of surfaces not having maximal variation: specifically, there exist surfaces with vertical vector fields (studied for example in \cite{CFGLS}).
\medskip

Let us now consider the case of a smooth projective hypersurface $X=V(f)$ in $\bP^n$: Beauville in \cite{Beau} proved an interesting equivalence between the maximal variation of $|\cO_X(1)|$ and the validity of a Lefschetz property in the Jacobian ring of $f$. In particular, he proves that $|\cO_X(1)|$ has maximal variation if and only if $R_f=\bC[x_0,\dots,x_n]/J_f=\bigoplus_iR_f^i$ (where $J_f$ is the Jacobian ideal of $f$) satisfies the Weak Lefschetz Property in degree $d-1$: this means that for a general element $x\in R^1_f$ the multiplication map $\cdot x:R^{d-1}_f\rightarrow R^d_f$ is injective. Such a property, as the Weak Lefschetz Properties in all the other degrees, is conjectured to hold for all the rings which are quotients of $\bC[x_0,\dots,x_n]$ over an ideal generated by a regular sequence of homogeneous polynomials (one can see \cite{bookLef} or \cite{JM}). As recalled in \cite{Beau2}, given a smooth hypersurface $X\subset\bP^n$ of degree $d\geq3$ with $n\geq3$, the maximal variation of $|\cO_X(1)|$ holds in the following cases:
\begin{itemize}
    \item if $X$ is of degree $d\geq n+2$ (partially proved in \cite{PRT} and then improved in \cite{BMR}),
    \item if $X$ is a quartic surface in $\bP^3$ (the corresponding weak Lefschetz property has been proved in \cite{BMMN}),
    \item if $X$ is a cubic threefold in $\bP^4$ (from the validity of the corresponding Lefschetz property, proved in \cite{AR}, as then in \cite{BFP}),
    \item if $X$ is general (the validity of the Lefschetz property, in any degree, is known to hold in the general case).
\end{itemize} 

However, infinitely many cases are still not covered by the results above: infinitely many hypersurfaces for which it is not known whether the above Lefschetz property or, equivalently, the maximal variation property holds. From our main theorem (Theorem A), we can improve the above results:
\begin{THC}[Theorem \ref{THM:hypersurfaces}]
    A smooth projective not Fano hypersurface $X$ in $\bP^n$ (i.e. such that $\deg(X)\geq n+1$), with $n\geq3$, has maximal variation.\\
    Moreover, if $X$ is a smooth complete intersection in $\bP^n$ of dimension at least $2$ of type $(d_1,\dots,d_r)$ such that $\sum_i d_i\geq n+1$, then $X$ has maximal variation.
\end{THC}
We thus contribute to give a positive answer to a question posed by Harris, Mazur, and Pandharipande in \cite{HMP}, dealing with the maximal variation of hyperplane sections of hypersurfaces. Moreover, by the equivalence proved by Beauville in \cite{Beau}, we obtain, as a corollary, the following:

\begin{THD}[Corollary \ref{COR:lefschetzproperties}]
    Let $X=V(f)\subset\bP^n$, with $n\geq3$, be a smooth hypersurface of degree $d=n+1$. Then the associated Jacobian ring $R_f=\bigoplus_{i=0}^{(n+1)(d-2)}R^i_f$ satisfies the Weak Lefschetz Property in degree $n=d-1$, i.e. if $x\in R^1_f$ is a general element then the multiplication map 
    $$\cdot x:R^{n}_f\rightarrow R^{n+1}_f \qquad {\mbox{is injective}}.$$
\end{THD}

Observe that the validity of the same Lefschetz property for the Jacobian ring of a smooth projective surface in $\bP^3$ of degree $d\geq4$ is already known as for that of a smooth hypersurface in $\bP^n$ of degree $d\geq n+2$, as recalled above. We find pleasant that the classical Lefschetz theory helps again (see for instance \cite{St}) in the algebraic Lefschetz problems.
\medskip
\ \\
The paper is organized as follows. In Section \ref{SEC:NotandPrelim} we introduce the main objects we will deal with and in Section \ref{SEC:pencilsandvariations} we prove Theorem A. In Section \ref{SEC:surfaces}, we focus on the case of surfaces, dealing with étale double coverings and Enriques surfaces (in Subsection \ref{SUBSEC:doublecoverings}) and with irregular surfaces of maximal Albanese dimension (in Subsection \ref{SUBSEC:irrsurfaces}), proving then Theorem B. In Section \ref{SEC:hypersurfaces} we consider the case of smooth projective hypersurfaces and related algebraic Lefschetz properties by proving Theorem C ad Theorem D. Finally, in Section \ref{SEC:openquestions}, we present some questions still open and naturally arising in this context.
\bigskip
\ \\
{\bf{Aknowledgments}}: The authors are members of GNSAGA (INdAM). The first named author is supported by INdAM and partially supported by INdAM-GNSAGA Project {\em{Surfaces of general type and their geometry}} CUP E53C25002010001. The authors want to express their gratitude to Alexander Kuznetsov who pointed out a mistake in a preliminary version of the work. The authors want also to warmly thank Arnaud Beauville for pointing out some useful remarks, as also the application of our main result to the case of smooth complete intersections.

\section{Notation and Preliminaries}
\label{SEC:NotandPrelim}

Let $X$ be a smooth complex projective variety of dimension $n\geq2$ and let $L$ be a very ample line bundle on $X$. Denoting by $V=H^0(X,L)$ the space of global sections of $L$, the line bundle $L$ induces an embedding $\phi:X\hookrightarrow \bP(V^*)$. By abuse of notation, we will identify $X$ with its image $\phi(X)$.\\
For $t\in\bP(V)$, we denote by $H_t$ the corresponding hyperplane and by $Y_t:=H_t\cap\phi(X)$ the hyperplane section of $X$ defined by $H_t$. Equivalently, $Y_t$ is the zero locus of a non-zero section $s_t\in V$, with $[s_t]=t$.
\smallskip
\ \\
There is a universal family parametrizing the hyperplane sections of $X$, namely $\pi:\cY\rightarrow\bP(V)$. More precisely, defining $\cY\subseteq\bP(V)\times X$ as
$$\cY:=\{(t,x): x \in Y_t\},$$
the first projection gives the above family $\pi$. For any subvariety $Z\subseteq\bP(V)$, we can then define the induced family
$$\pi_Z:\cY_Z:=\pi^{-1}(Z)\rightarrow Z.$$
Note that if $\dim(Z)>0$, then the second projection 
$$\psi_Z:\cY_Z\rightarrow X$$
is dominant.

\begin{definition}
\label{DEF:isotrivial}
    Let $Z\subseteq\bP(V)$ be a connected subvariety with $\dim(Z)>0$. We say that
    \begin{enumerate}
        \item $Z$ is {\textnormal{isotrivial}} if the family of hyperplane sections $\{Y_z\}_{z\in Z}$ is a constant family, or, equivalently, if the induced family $\cY_Z\rightarrow Z$ is isotrivial on an open subset of $Z$;
        \item $Z$ is {\textnormal{birationally isotrivial}} if the hyperplane sections $\{Y_z\}_{z\in Z}$ are all birational to one another.
    \end{enumerate}
\end{definition}

\begin{definition}
    \label{DEF:maximalvariationofL}
    Given a smooth complex projective variety $X$ and a very ample line bundle $L$ on it, we say that the linear system $|L|$ has {\textnormal{maximal variation}} if there does not exist an isotrivial subvariety $Z\subseteq\bP(H^0(X,L))$ passing through the general point of $\bP(H^0(X,L))$.\\
    Moreover, we say that $X$ has {\textnormal{maximal variation}} if every very ample line bundle $L$ on it does.
\end{definition}

Let now $W\subseteq\bP(V)$ be the open subset corresponding to smooth hyperplane sections of the smooth complex projective variety $X$. Given $t\in W$ and the corresponding $Y_t$, which is then a smooth projective variety of dimension $n-1$, denote by 
$$j_t:Y_t\hookrightarrow X$$
the natural inclusion in $X$.\\
By Lefschetz theorem on hyperplane sections (see for example \cite{Voi2}[Theorem 1.23]), we have that 
$$H^k(X) \xrightarrow{\sim} H^k(Y_t) \qquad \mbox{for} \quad k<n-1,$$
while in degree $n-1$ we have just an injection $j_t^*:H^{n-1}(X)\hookrightarrow H^{n-1}(Y_t)$. Here, the so-called {\it{vanishing cohomology}} appears (\cite{Voi2}[Ch. 2]). 
\begin{definition}
    Let $Y$ be a hyperplane section of dimension $n-1$ of a smooth projective variety $X$ and $j:Y\hookrightarrow X$ the natural inclusion. For every coefficient ring $A$, the vanishing cohomology in degree $n-1$ of $Y$ is defined as the kernel of the Gysin morphism:
    $$H^{n-1}(Y,A)_{\mbox{van}}:=\ker(j_*:H^{n-1}(Y,A)\rightarrow H^{n+1}(X,A)).$$
\end{definition}

Denoting by $\bH_t$ the vanishing cohomology in degree $n-1$ (with rational coefficients) of $Y_t$, i.e. $H^{n-1}(Y_t,\bQ)_{\mbox{van}}$, we have a decomposition as an orthogonal direct sum, with respect to the intersection form on $H^{n-1}(Y_t,\bQ)$:
\begin{equation}
\label{EQ:vanishing}
H^{n-1}(Y_t,\bQ)=\bH_t\oplus j_t^*H^{n-1}(X,\bQ),
\end{equation}
where the second summand gives a constant Hodge structure.\\
Let us now recall the following:
\begin{definition}
        Given a holomorphic line bundle $L$ on a complex projective variety $X$, a pencil of hypersurfaces on $X$ is a projective line $\bP^1\subset \bP(H^0(X,L))$. Such a pencil is a Lefschetz pencil if it is such that 
        \begin{itemize}
            \item its base locus is smooth of codimension $2$ in $X$,
            \item the general hypersurface is smooth and the singular ones have at most one ordinary double point as singularity.
        \end{itemize}
\end{definition}

For what concerns the existence of such a pencil, it is well known (see, for example, \cite{Voi2}[Cor. 2.10]) that if $X$ is a smooth projective complex variety, then a generic pencil of hyperplane sections of $X$ is a Lefschetz pencil.

\begin{remark}
    The vanishing cohomology $\bH_t$ of the hyperplane section $Y_t$ is contained in the primitive cohomology of degree $n-1$ of $Y_t$ and it is generated by the classes of vanishing spheres of a Lefschetz pencil passing through $t\in \bP(V)$ {\textnormal{(\cite{Voi2}[Cor. 2.25, Lem. 2.26])}}.
\end{remark}

Let us fix a Lefschetz pencil $\bL\subseteq\bP(V)$ and let $\bL_0\subseteq\bL$ be the open set corresponding to the smooth hyperplane sections (i.e. $\bL_0=\bL\cap W$). Fix a general point $0\in\bL_0$ and let $Y=Y_0$ be the general hyperplane section of $X$ parametrized by $\bL$. 
\begin{remark}
\label{REM:irreduciblerepresentation}
    Recall that the monodromy action induces an irreducible representation of the fundamental group $\pi_1(\bL_0,0)$ on the vanishing cohomology $\bH_0$ {\textnormal{(\cite{Voi2}[Th. 3.27])}}.
\end{remark}

\section{Lefschetz pencils and variations}
\label{SEC:pencilsandvariations}

With the previous notation, in this section we state and prove our main result:
\begin{theorem}
\label{THM:main}
    Let $X$ be a smooth projective variety of dimension $\dim(X)=n\geq2$, embedded in $\bP(V^*)$, where $L$ is a very ample line bundle on $X$ and $V=H^0(X,L)$.\\
    Assume that
        $$p_g(X)=\dim(H^{n,0}(X))>0 \qquad \mbox{and} \qquad H^{n-1,0}(X)=\{0\}.$$
    Then there doesn't exist an isotrivial subvariety of $\bP(V)$ passing through the general point of $\bP(V)$. \\
    In particular, $|L|$ has maximal variation and so $X$ has maximal variation.
\end{theorem}

\begin{proof}
    Let us proceed by contradiction: we assume that an isotrivial subvariety $Z\subseteq\bP(V)$ of positive dimension passes through the general point $0$ of $\bP(V)$ (recall Definition \ref{DEF:isotrivial}). In particular, we will consider then a Lefschetz pencil $\bL$ passing through $0$ (hence, $0$ denotes also the general point of the considered Lefschetz pencil).\\
    Our aim is to construct a curve $B\rightarrow \bL\simeq\bP^1$, parametrizing $1$-dimensional isotrivial families of hyperplane sections of $X$.\\
    To do that, let us assume, in particular, that there exists an irreducible isotrivial curve passing through $0\in\bL$. Such a curve deforms by moving away from the general point $0$, and by spreading the process, we can define a surface $\cD\subset\bP(V)$ such that 
    \begin{itemize}
        \item $\cD$ contains the Lefschetz pencil $\bL$,
        \item $\cD$ is covered by a $1$-dimensional family of isotrivial curves,
        \item there are finitely many of these curves passing through the general point of $\bL$.
    \end{itemize}
    By considering the corresponding universal family parametrized by a curve $B'$ contained in a suitable Hilbert scheme of $\bP(V)$, we can define the map
    $$q':\cS'\rightarrow B',$$
    such that for a general point $b'\in B'$ the fiber $q'^{-1}(b')$ is an isotrivial curve contained in $\cD\subset\bP(V)$. Observe that up to a normalization at the general point, we can assume that the general fiber of the map $q'$ is a smooth curve.\\
    Considering the natural dominant map $p':\cS'\rightarrow\cD$, let us denote by $\cL$ a component of $p'^{-1}(\bL)$ dominating $B'$ via $q'$. By the base change $B\rightarrow B'$, where $B$ is the normalization of $\cL$, we can then construct a family 
    $$q:\cS\rightarrow B,$$
    which is endowed with the section naturally arising from the considered base change.\\
    By using the notations introduced in Section \ref{SEC:NotandPrelim}, the surface $\cD$ defines a family $\cY_{\cD}$ of hyperplane sections of $X$, inducing, via the dominant map $p:\cS\rightarrow \cD$, a family
    $$\pi:p^*(\cY_{\cD})=:\cY_{\cS}\rightarrow \cS,$$ where $\dim(\cY_{\cS})=\dim(\cY_{\cD})=n+1$. Given a general point $b\in B$, we have then
    $$\cY_{\cS_b}\xrightarrow{\pi_b}\cS_b\xrightarrow{q}\{b\},$$
    where $\cS_b$ represents an isotrivial curve and $\cY_{\cS_b}$ the corresponding induced family of hyperplane sections of $X$ ($\dim(\cY_{\cS_b})=n$). \\
    Putting this construction in family again, we have then obtained
    $$
    \begin{tikzcd}
    & \cS \arrow[dr, "q"]
       & \\
    \cY_{\cS} \arrow[ur]
      \arrow[rr, "\pi"]
      \arrow[d, "\psi"']
    && B^0
      \arrow[ul, bend right=35, "\iota"']
      \arrow[d] \\
    X && \mathbb{P}^1
    \end{tikzcd}
    $$
    where $B^0$ denotes a suitable Zariski open subset of $B$. In particular, for a point $b\in B^0$, after performing a base change and up to a birational map, we can describe the family $\cS_b$ as a product $C_b\times Y_b$ where $C_b$ is a curve and $Y_b$ is the general element of the family $\cS_b$ (see \cite{BBG}). We can then define the dominant rational (and generically finite) map
    $$\psi_b:C_b\times Y_b\dashrightarrow X.$$
    Since, by assumption, $p_g(X)>0$ we can consider the injective non-trivial morphism of Hodge structures, defined as the pullback of top forms:
    \begin{equation}
        \label{EQ:psibstar}
        \psi_b^*:H^{n,0}(X)\rightarrow H^{n,0}(C_b\times Y_b).
    \end{equation}
    To analyze the image of $\psi_b^*$, let us start by decomposing the cohomology group $H^n(C_b\times Y_b)$ via K\"unneth formula:
    $$H^n(C_b\times Y_b)\cong\left(H^0(C_b)\otimes H^n(Y_b)\right)\oplus\left(H^1(C_b)\otimes H^{n-1}(Y_b)\right)\oplus\left(H^2(C_b)\otimes H^{n-2}(Y_b)\right).$$
    Since $H^0(C_b)\cong\bQ$ and $H^2(C_b)\cong\bQ$, we have, by the Lefschetz hyperplane theorem,
    $$H^n(C_b\times Y_b)\cong H^{n-2}(X)^{\oplus2}\oplus \left(H^1(C_b)\otimes H^{n-1}(Y_b)\right).$$
    From Equation \eqref{EQ:vanishing}, 
    we can write 
    $$H^n(C_b\times Y_b)\cong H^{n-2}(X)^{\oplus2}\oplus\left(H^1(C_b)\otimes j_b^*H^{n-1}(X)\right)\oplus\left(H^1(C_b)\otimes\bH_b\right),$$
    where $j_b$ is the natural inclusion of $Y_b$ in $X$.
    Since we are looking at the level of forms of type $(n,0)$ and since, by assumptions, $H^{n-1,0}(X)=\{0\}$, it is then clear that the image of $H^{n,0}(X)$ via $\psi_b^*$ has to lie in the summand $H^1(C_b)\otimes\bH_b$. In particular, we are then looking at the injective morphism
    $$\psi^*_b:H^{n,0}(X)\rightarrow H^{1,0}(C_b)\oplus\bH_b^{n-1,0},$$
    where, by abuse of notation, we keep on denoting by $\bH_b$ the vanishing cohomology with $\bC$-coefficients.
    Let then $\omega\in H^{n,0}(X)$ be a non zero class: by the injectivity of $\psi^*_b$ we must have
    $$0\neq\psi_b^*(\omega)\in H^{1,0}(C_b)\otimes\bH_b^{n-1,0}.$$
    Let us now fix $\{\alpha_i\}_i$ a basis of $H^{1,0}(C_b)$ e $\{\beta_j\}_j$ a basis of $\bH_b^{n-1,0}$.    We can then write
    $$\psi_b^*(\omega)=\sum_{i,j}a_{i,j}\alpha_i\wedge\beta_j,$$
    for some complex coefficients $a_{i,j}$: our claim is to show that the coefficients $a_{i,j}$ are zero for every $i$ and $j$, to get that $\psi_b^*(\omega)$ has to be actually trivial, getting the desired contradiction.\\
    To do so, let us use the monodromy action and the Picard-Lefschetz formula (see for example \cite{Voi2}). We take a loop around a critical point of the pencil and lift it to $B$ (with some multiplicity). Let then $\delta$ be the corresponding vanishing cycle around one of the inverse images of the critical point and $T:=T_{\delta}$ the associated monodromy action. As a first fact, let us observe that $T$ fixes the class $\psi_b^*(\omega)$, since it is defined over the variety $X$. It then gives
    \begin{equation}
    \label{EQ:Tomega}
        \psi_b^*(\omega)=T(\psi_b^*(\omega))=\sum a_{i,j}T(\alpha_i)\wedge T(\beta_j)=\sum a_{i,j}T(\alpha_i)\wedge(\beta_j+k_{\delta}(\delta\cdot\beta_j)\delta),
    \end{equation}
    where $k_{\delta}\neq0$ is an integer depending on the ramification of the base change $B\rightarrow \bL$ around the critical point of the Lefschetz pencil corresponding to the vanishing cycle $\delta$. Writing $T(\alpha_i)=\alpha_i+S(\alpha_i)$, we get
    $$\psi_b^*(\omega)=\sum a_{i,j}(\alpha_i+S(\alpha_i))\wedge(\beta_j+k_{\delta}(\delta\cdot\beta_j)\delta)=$$
    $$=\sum a_{i,j}\alpha_i\wedge\beta_j+\sum a_{i,j}S(\alpha_i)\wedge\beta_j+\left(\sum k_{\delta}a_{i,j}(\delta\cdot\beta_j)T(\alpha_i)\right)\wedge\delta,$$
    where the first summand of this last expression equals $\psi(\omega)$. Hence, we have
    \begin{equation}
        \label{EQ:ThetaeGamma}
        0=\sum a_{i,j}S(\alpha_i)\wedge\beta_j+\left(\sum k_{\delta}a_{i,j}(\delta\cdot\beta_j)T(\alpha_i)\right)\wedge\delta=:\Theta+\Gamma\wedge\delta,
    \end{equation}
    where $\Theta\in H^{1,0}(C_b)\otimes\bH_b^{n-1,0}$, $\Gamma\in H^{1,0}(C_b)$, and $\delta$ is a real class.\\
    First of all, let us assume that $\Gamma\neq0$: we get that $\Theta\neq0$ and, in particular, $\Gamma\wedge\delta=-\Theta$. Hence, we obtain that $\Gamma\wedge \delta\in H^{1,0}(C_b)\otimes\bH_b^{n-1,0}$, i.e. the component $\delta^{n-1,0}$ has to be different from zero. By conjugation, since it is a real class, we have that also $\delta^{0,n-1}\neq0$. Observe that, since $T$ is a linear isomorphism, the elements $T(\alpha_i)$ give again a basis of $H^{1,0}(C_b)$, and so $\Gamma\wedge\delta^{0,n-1}\neq0$: comparing Hodge types, we would have that $\Theta + \Gamma\wedge\delta$ can not be equal to zero, giving a contradiction. 
    The only possibility is then that $\Gamma=0$, i.e.
    $$\sum k_{\delta}a_{i,j}(\delta\cdot\beta_j)T(\alpha_i)=0.$$
    Since $T$ is a linear isomorphism, the elements $T(\alpha_i)$ are independent and, since $k_{\delta}\neq0$, for each index $i$, we get
    \begin{equation}
        \label{EQ:irreducibilitystep}
        \delta\cdot\sum_j a_{i,j}\beta_j=0.
    \end{equation}
    Observing that this holds for any vanishing cycle $\delta$ and recalling that these vanishing cycles generate $\bH_b$ and that on $\bH_b$ the intersection is non degenerate, we get
    $$\sum_j a_{i,j}\beta_j=0.$$
    Finally, since the $\beta_j$'s are independent, we get that $a_{i,j}$ has to be zero for every $i$ and every $j$ as claimed, i.e.
    $$\psi_b^*(\omega)=0,$$
    which gives a contradiction.
\end{proof}


\begin{remark}
    From the results proved in \cite{BBG}, one can actually repeat the above proof showing also that, with the same hypotheses, there does not exist even a birationally isotrivial subvariety (see Definition \ref{DEF:isotrivial}), passing through the general point of $\bP(V)$. 
\end{remark}

As a first application of the above result (other consequences will be discussed in the following sections), let us state the following:
\begin{corollary}
    Any smooth hyperk\"ahler variety has maximal variation.
\end{corollary}

As observed by Beauville \cite{Beau2}[Proposition 1], the maximal variation of the linear system $|L|$ for a line bundle $L$ on a smooth projective variety $X$ is equivalent to the vanishing of the cohomological group $H^0(Y,T_{X|Y})$ for $Y\in|L|$ general. Indeed, the coboundary map 
$$\delta: H^0(Y,N_{Y/X})\rightarrow H^1(Y,T_Y)$$
can be interpreted as the differential of a modular map from the Hilbert scheme of $X$ to a Kuranishi family.
By Theorem \ref{THM:main}, we then obtain the following result:
\begin{corollary}
    Let $X$ be a smooth projective variety of dimension $n\geq2$ with $p_g(X)>0$ and $H^{n-1,0}(X)=\{0\}$ and $L$ be a very ample line bundle on $X$. Then, for $Y$ being a general element of $|L|$, the map
    $H^0(Y,N_{Y/X})\rightarrow H^1(Y,T_Y)$ is injective. In particular, $H^0(Y,T_Y)$ and $H^0(Y,T_{X|Y})$ are isomorphic.
\end{corollary}

Let us conclude this section by stressing that in the case where $X$ is a surface, by the above Theorem \ref{THM:main}, we get for $X$ the maximal variation property by requiring that it has positive geometric genus and $q(X)=0$, even though this may appear somewhat counterintuitive. Irregular surfaces, indeed, exhibit in general much greater rigidity and therefore more chances for the maximal variation property to be satisfied, as in the case of abelian surfaces, for which such a property is known (see for example \cite{Beau2}). We'll partially treat this case in the next section.

\section{The case of surfaces}
\label{SEC:surfaces}
In this section, let us consider the case of surfaces. Let $X$ be a smooth surface with $p_g(X)=h^2(X,\cO_X)>0$ and $L$ be, as above, a very ample line bundle on $X$, with $V=H^0(X,L)$. Let again $W$ be the open subset of $\bP(V)$ corresponding to smooth curves, sections of $X$. If $t\in W$, the genus of the corresponding section $Y_t$ is given, from adjunction formula, by
$$g=g(Y_t)=\frac{(K_X+L)\cdot L}{2}+1.$$
We can then consider a modular map $m: W\rightarrow \cM_g$, where $\cM_g$ is the moduli space of curves of genus $g$. Assuming that $p_g(X)>0$ and that $q(X)=0$, observe that we also get $g\geq 3$, as $L$ is very ample. From Theorem \ref{THM:main}, we have:
\begin{corollary}
\label{COR:genfiniteforsurfaces}
    Assume that $L$ is a very ample line bundle on a surface $X$ with $q(X)=0$ and $p_g(X)>0$. Then the modular rational map $\bP(H^0(X,L))\dashrightarrow\cM_g$ is generically finite.
\end{corollary}

In particular, applying the above result to a K3 surface we get the result of Dutta and Huybrechts (see \cite{DH}) and then of Bakker (see \cite{Bak}) (for the case of very ample line bundles). On the other hand, the case of Enriques surfaces, for which the hypotheses of Corollary \ref{COR:genfiniteforsurfaces} are not satisfied, seems not to be treatable with our techniques. Nevertheless, in what follows we get the same result as above also for Enriques surfaces, by lifting the problem to the associated K3 surface.

\subsection{Enriques and double coverings}
\label{SUBSEC:doublecoverings}
\ \\
Let us now look at those varieties $S$ of dimension $\dim(S)=n\geq2$, such that $h^{n,0}(S)=p_g(S)$ is zero and admitting an étale $2:1$ cover $\pi:X\rightarrow S$ such that $p_g(X)\neq0$. Our aim is to obtain the same result as the one of Theorem \ref{THM:main} for the variety $S$, by lifting the problem at the level of $X$.\\
In particular, we focus on the case of Enriques surfaces (partially treated in \cite{CDGK}) and we prove the following (see Theorem B from the Introduction):
\begin{theorem}
\label{THM:Enriques}
    Let $S$ be any smooth Enriques surface and $L$ be a very ample line bundle on $S$. Then $|L|$ has maximal variation and so $S$ has maximal variation.
\end{theorem}

To prove such a result, let us start with a couple of facts. 

\begin{fact}
\label{REM:veryampleness}
    In {\textnormal{\cite{LP}[Lemma 1.1]}}, it is proved that if $\pi:\tilde{Y}\rightarrow Y$ is a double cover ramified at a smooth divisor $B$, with $M^{\otimes2}=\cO_Y(B)$ and if $H$ is a very ample line bundle on $Y$, then the line bundle $\pi^*(H)$ is very ample on $\tilde{Y}$ if and only if $H\otimes M^{-1}$ is generated by global sections on $X$.\\
    In the specific case of an Enriques surface $S$, the $2:1$ cover $\pi$ from the corresponding $K3$ surface $X$ is étale and determined by the canonical line bundle $K_S$ which is a $2$-torsion element. If $L$ is a very ample line bundle on $S$, then by Reider's Theorem (see {\textnormal{\cite{Rei}}}, or {\textnormal{\cite{BHPV}[Thm. 11.4]}}) it's easy to show that $L\otimes K_S$ it's generated by its global sections. Hence, by the aforementioned result, the pull-back on the $K3$ surface $X$ of a very ample line bundle is again very ample. 
\end{fact}

\begin{fact}
    \label{REM:isotriviality}
    Consider the embedding $S\rightarrow\bP(V^*)$ via a very ample line bundle $L$, such that $V=H^0(S,L)$, and also an isotrivial curve $\cC'$ in $\bP(V)$ (see Definition \ref{DEF:isotrivial}). By the above Fact \ref{REM:veryampleness} and since the double cover $\pi:X\rightarrow S$ is étale, the pull-back $\cC$ of $\cC'$ via $\pi$ defines again an isotrivial curve for the $K3$ surface $X$.
\end{fact}

Let us now prove the above Theorem \ref{THM:Enriques}:
\begin{proof}
    Let us denote by $X$ the K3 surface which is the étale $2:1$ cover of the Enriques surface $S$, via $\pi:X\rightarrow S$ and let $\iota$ be the associated involution.\\
    Let us proceed by contradiction. As done in the proof of Theorem \ref{THM:main}, we can put ourselves in the situation where there exists a dominant rational map $\cC'\times Z\dashrightarrow S$, where $\cC'$ is an isotrivial curve and $Z$ is the hyperplane section defined by a general point of $\cC'$. From Fact \ref{REM:isotriviality}, we have an induced isotrivial curve $\cC$ for the $K3$ $X$ and so a dominant rational map
    $$\psi:\cC\times Y\dashrightarrow X,$$
    where $Y:=\pi^{-1}(Z)$. Observe here that the general hyperplane section parametrized by $\cC$ is smooth (and birational to $Y$) and that the singular fibers are in correspondences of those singular sections for the Enriques surface $S$. Moreover, because of the double cover $\pi$, these special fibers are singular in two distinct nodes. Observe, moreover, that the curve $\cC$ is either birational to $\cC'$ or it is a double cover over $\cC'$.\\
    Since the geometric genus $p_g(X)$ of $X$ is not zero, the dominant map $\psi$ induces an injective and non-trivial morphism at the level of top forms, and in particular we get
    $$\psi^*:H^{2,0}(X)\rightarrow H^{2,0}(\cC\times Y).$$
    We now proceed as in the proof of Theorem \ref{THM:main}, with a particular attention on the nature of the vanishing cohomology for the sections of the double cover $X$: to do so, we refer to \cite{LP}.\\
    Observe that $H^1(X,\cO_X)=0$ and that the $(2,0)$-form $\omega$ generating $H^{2,0}(X)$ is anti-invariant with respect the involution $\iota$ on $X$. Then from K\"unneth decomposition, we have
    \begin{equation}
        \label{EQ:psistarEnriques}
        \psi^*:H^{2,0}(X)\rightarrow \left(H^{1,0}(\cC)^+\otimes (\bH^{1,0})^-\right)\oplus \left(H^{1,0}(\cC)^-\otimes (\bH^{1,0})^+\right),
    \end{equation}
    where $\bH$ denotes the vanishing cohomology of $Y$ and $\bH^-$ and $H^{1,0}(\cC)^-$ (respectively, $\bH^+$ and $H^{1,0}(\cC)^+$) denotes its anti-invariant (respectively, invariant) part with respect to the corresponding involutions.\\
    The image $\psi^*(\omega)$ can then a priori split into the above summands: let us firstly focus, with an abuse of notation, on
    $$\psi^*:H^{2,0}(X)\rightarrow H^{1,0}(\cC)^+\otimes (\bH^{1,0})^-.$$
    We can then write
    $$0\neq \psi^*(\omega)=\sum_{i,j}a_{ij}\alpha_i\wedge\beta_j,$$
    where $\{\alpha_i\}_i$ is a basis of $(H^{1,0}(\cC))^+$ and $\{\beta_j\}_j$ a basis of $(\bH^{1,0})^-$.\\
    In \cite{LP}, it has been showed that $\bH^-$ is the subspace of the anti-invariant primitive cohomology generated by the cycles $\gamma_{1s}-\gamma_{2s}$, where the $\gamma_{is}$ are the vanishing cycles of the two nodes over a critical point of the considered Lefschetz pencil $\bL$, i.e. the Lefschetz pencil which $\cC$ intersects in the general point.\\
    Let us now consider a critical point $s$ of $\bL$ and the corresponding monodromy action $T:=T_s$. As done in the proof of Theorem \ref{THM:main} (see Equation \eqref{EQ:Tomega}), we can write
    $$\psi^*(\omega)=T(\psi^*(\omega))=\sum_{i,j}a_{ij}T(\alpha_i)\wedge T(\beta_j).$$
    In this case, by the Picard-Lefschetz formula, we can write
    $$T(\beta_j)=\beta_j+k_s(\beta_j\cdot\gamma_{1,s})\gamma_{1s}+k_s(\beta_j\cdot\gamma_{2s})\gamma_{2s},$$
    where the integer $k_s\neq0$ is a suitable constant. As observed in \cite{LP}, since $\beta_j\in\bH^-$, we have that
    $$0=(\beta_j\cdot\gamma_{1s}+\gamma_{2s})=(\beta_j\cdot\gamma_{1s})+(\beta_j\cdot\gamma_{2s}) \quad \Longrightarrow \quad (\beta_j\cdot\gamma_{2s})=-(\beta_j\cdot\gamma_{1s}).$$
    Moreover, we have
    $$(\beta_j\cdot\gamma_{1s}-\gamma_{2s})=(\beta_j\cdot\gamma_{1s})-(\beta_j\cdot\gamma_{2s})=2(\beta_j\cdot\gamma_{1s}) \quad \Longrightarrow \quad (\beta_j\cdot\gamma_{1s})=\frac{1}{2}(\beta_j\cdot \gamma_{1s}-\gamma_{2s}).$$
    Hence, by denoting $\tilde{\gamma}_s:=\gamma_{1s}-\gamma_{2s}$, we can write
    $$T(\beta_j)=\beta_j+\frac{k_s}{2}(\beta_j\cdot\tilde{\gamma}_s)\tilde{\gamma}_s.$$
    By writing $T(\alpha_i)=\alpha_i+S(\alpha_i)$, as in the proof of Theorem \ref{THM:main} (see Equation \eqref{EQ:ThetaeGamma}), we get
    $$0=\sum_{i,j}S(\alpha_i)\wedge\beta_j + \left(\sum_{i,j}\frac{k_s}{2}a_{ij}(\beta_j\cdot\tilde{\gamma}_s)T(\alpha_i)\right)\wedge\tilde{\gamma}_s=\Theta+\Gamma\wedge\tilde{\gamma}_s.$$
    Assume that $\Gamma\neq0$: since $-\Gamma\wedge\tilde{\gamma}_s=\Theta$, we get by construction that $\tilde{\gamma}_s\in(\bH^{1,0})^-$, which is not possible, since it is a real class. We then get that $\Gamma=0$ and we obtain a contradiction as in the proof of Theorem \ref{THM:main}. Indeed, we get (as in Equation \eqref{EQ:irreducibilitystep}) that for any index $i$ 
    $$\tilde{\gamma}_s\cdot \sum_j a_{ij}\beta_j=0.$$
    Since in \cite{LP}, the authors also showed that the action of the monodromy of the Lefschetz pencil on $\bH^-$ is irreducible, one can then conclude that the coefficients $a_{ij}$ are zero for any pair of indices $i$ and $j$, getting a contradiction.\\
    In other words, we have shown that the map $\psi^*$ in \eqref{EQ:psistarEnriques} has image in the second summand, namely in $H^{1,0}(\cC)^-\otimes (\bH^{1,0})^+$. Observing that $(\bH^{1,0})^+=H^{1,0}(Z)$, one can get a contradiction reasoning in the same way as done in the proof of Theorem \ref{THM:main}, ending the proof.
\end{proof}


\subsection{Irregular surfaces}
\label{SUBSEC:irrsurfaces}
\ \\
In this subsection, we extend the previous results also to the case where $X$ is an irregular surface with Albanese dimension $2$. 

Let us start with a simple result.
\begin{lemma} 
\label{LEM:torsion}
Let $C$ be a smooth curve and $A$ be an abelian variety of dimension $q\geq2$. Let $f:C\rightarrow A $ be a map such that $f(C)=Z$ generates $A$ as a group. For $a\in A$ set $Z_a=\{x\in A: x=z+a, z\in Z\}.$\\
If $Z_a$ is rationally equivalent to $Z$ then $a$ is torsion point of $A$. 
\end{lemma}

\begin{proof} 
Fix a subvariety $D$ in $A$ of dimension $q-2$ such that $\Theta=Z+D= \{x\in A: x=z+d, z\in Z,  d\in D\}$
is an ample divisor (e.g. $D=Z+Z+\dots Z$ $q-2$ times). Then, by assumption, $\Theta$ is rationally equivalent to $Z_a+D= \Theta_a$.
By setting $L=\cO_A(\Theta)$, we obtain (see for example \cite{Mum}) that $a\in K(L)$, where $K(L)$ is the kernel of the map induced by $L$ $\phi_L:A\to Pic^0(A)$ 
$$\phi_L(x)=\cO_A(\Theta_x-\Theta).$$ 
Since $L$ is ample, we have that $\phi_L$ is an isogeny and so its kernel is finite. Then, as claimed, we get that $a$ is a torsion point of $A$. 
\end{proof}

Let now $X$ be an irregular surface with $q(X)\geq2$. Let, moreover, $A$ be the Albanese variety of $X$ and $\alpha:X\rightarrow A$ be the Albanese map. Finally, let us denote by $B=\alpha(X)$ the image of $\alpha$.

\begin{proposition}
\label{PROP:maxvarirrsurf}
Assume that $B=\alpha(X)$ is a surface in $A$ and let $H$ be a line bundle on $X$. Assume that, for $Y\in|H|$, $\alpha (Y)$ generates $A$, and that the general element of $|H|$ is irreducible. Then $|H|$ does not contain a birationally isotrivial curve, whose general element is irreducible.
\end{proposition}

\begin{proof} 
Assume by contradiction that there exists a birationally isotrivial family in $|H|$, whose general element is irreducile. Then there exist a smooth curve $C$ (the general element of such a family) and a parametrizing curve $T$, in such a way that we can define (by the result in \cite{BBG}) a rational map $\Psi: T\times C\dashrightarrow X$ such that for $t\in T$ general
$$\Psi(t\times C)=Y_t\in |H|.$$ 
Composing then with the Albanese map $\alpha$, we obtain a morphism $\varphi:T\times C\rightarrow A$ with image contained in $B$. One can see that we indeed obtain an isotrivial family of translated curves $\alpha(C)_t$ covering $B$. Since the curves $\alpha (Y)_t$ are all rationally equivalent to each other, from Lemma \ref{LEM:torsion} we get that the general point $t$ in $T$ is a torsion point of $A$ and, in particular, an element of the kernel of a fixed isogeny. In other words, $t$ belongs to a fixed finite subgroup of $A$, yielding a contradiction.
\end{proof}

Since a very ample line bundle $H$ on $X$, with $X$ as above, satisfies the assumptions of the above Proposition \ref{PROP:maxvarirrsurf}, we get in particular the following (see Theorem B from the Introduction):
\begin{corollary}
\label{COR:maxvarirrsurf}
    A smooth projective surface $X$ with maximal Albanese dimension has maximal variation.
\end{corollary}

Finally, one can see that the same argument applies even in higher dimension:
\begin{corollary}
    A smooth projective variety with maximal Albanese dimension has maximal variation.
\end{corollary}

\section{Hypersurfaces and Lefschetz properties}
\label{SEC:hypersurfaces}

Let us now focus on the case where $X$ is a smooth projective hypersurface of degree $d>2$ in $\bP^n$ for $n\geq3$. 
From Theorem \ref{THM:main}, we are able to give a new improvement to a question posed by Harris, Mazur, and Pandharipande in \cite{HMP}:
\begin{question}
\label{QUE:question}
    Let $X$ be a smooth hypersurface of degree $d$. Does the family of all smooth hyperplane sections of $X$ vary maximally in moduli? In other words, does the linear system $|\cO_X(1)|$ have maximal variation?
\end{question}

Indeed, one immediately gets by the adjunction formula the following (Theorem C from Introduction):
\begin{theorem}
    \label{THM:hypersurfaces}
    A smooth projective not Fano hypersurface $X$ in $\bP^n$ for $n\geq3$, i.e. any smooth hypersurface in $\bP^n$ with degree $\deg(X)\geq n+1$, has maximal variation.\\
    Moreover, if $X$ is a smooth complete intersection in $\bP^n$ of dimension at least $2$ of type $(d_1,\dots,d_r)$ such that $\sum_i d_i\geq n+1$, then $X$ has maximal variation.
\end{theorem}

As recalled in the Introduction, such a result was already known for smooth hypersurfaces in $\bP^n$ of degree $d\geq n+2$ by \cite{BMR}.\\

Let us conclude by considering the validity of specific Lefschetz properties for Jacobian rings of hypersurfaces. Let us start by recalling the following:

\begin{definition}
\label{DEF:jacobianring}
    Let $f\in\bK[x_0,\dots,x_n]$ be a homogeneous polynomial of degree $d$ defining a smooth hypersurface in $\bP^n$. The Jacobian ring $R_f$ of $f$ is defined as the quotient 
    $$R_f=\bK[x_0,\dots,x_n]/J_f=R_f^0\oplus\dots,\oplus R_f^N,$$
    where $J_f$ is the Jacobian ideal of $f$, i.e. the ideal generated by the first partial derivatives of $f$ with respect to the $x_i$'s.
\end{definition}

For Jacobian rings of smooth hypersurfaces (as, actually, for a larger class of algebras, namely for the standard Artinian Gorenstein algebras) one can define the following properties (the interested reader can refer for example to \cite{bookLef} or also the recent \cite{JM})):

\begin{definition}
    We say that the Jacobian ring $R_f$ of a smooth hypersurface $V(f)$ satisfies the Weak Lefschetz Property in degree $k$, if there exists an element $x \in R^1_f$ such that the multiplication map 
    $$\cdot x: R^k_f\rightarrow R^{k+1}_f$$
    is of maximal rank.
\end{definition}

In \cite{Beau2} (as in \cite{Beau}, for the case of cubic threefolds), Beauville proved the following equivalence:

\begin{proposition}
\label{PROP:beauville}
    Let $X=V(f)\subset\bP^n$ with $n\geq3$ be a smooth hypersurface of degree $d\geq3$ (or $d\geq4$ if $n=3$). The linear system $|\cO_X(1)|$ has maximal variation if and only if the Jacobian ring $R_f$ of $X$ satisfies the weak Lefschetz property in degree $d-1$.
\end{proposition}

\begin{remark}
\label{REM:cubicthreefolds}
    From the above equivalence (Proposition \ref{PROP:beauville}), Beauville observed that the maximal variation of $|\cO_X(1)|$ for a smooth cubic threefold $X=V(f)\subset\bP^4$ follows then by the validity of the Weak Lefschetz Property in degree $2$ for the Jacobian ring $R_f$, proved \cite{AR}, and see also a later proof in \cite{BFP}.\\
    Moreover, in \cite{BMR} it is proved the validity of the above Weak Lefschetz Property in degree $d-1$ for example for hypersurfaces in $\bP^n$ of degree $d\geq n+2$ (as also Lefschetz properties in other degrees), getting the corresponding maximal variation for the hyperplane linear system. Finally, also the maximal variation property with respect to the linear system $|\cO_X(1)|$ has been solved for the case of surfaces in $\bP^3$: in this case the validity of the corresponding weak Lefschetz property has indeed been proved in \cite{BMMN}. 
\end{remark}

From Theorem \ref{THM:hypersurfaces}, by Beauville's result, we get the following corollary (see Theorem D from Introduction):

\begin{corollary}[Algebraic Weak Lefschetz Property]
    \label{COR:lefschetzproperties}
    \ \\
    Let $X=V(f)\subset\bP^n$, with $n\geq3$, be a smooth hypersurface of degree $d\geq n+1$. Then the Jacobian ring $R_f=\bigoplus_{i\geq0}R_f^i$ satisfies the Weak Lefschetz Property in degree $d-1$.
\end{corollary}

\section{Open questions}
\label{SEC:openquestions}

To conclude, we present some questions arising from the study carried out in the previous sections.

\begin{enumerate}
    \item Given a smooth projective variety $X$ and a very ample line bundle $L$ on $X$, it should be interesting to study the {\it{isotrivial locus}} of $\bP(H^0(X,L))$, i.e. the locus covered by isotrivial subvarieties of positive dimension (in the sense of Definition \ref{DEF:isotrivial}). For example, one could wonder if there exists an upper bound for its dimension, which we actually expect to be small.
    \item Another natural problem is the computation of the degree of the associated modular maps. Some examples where such a degree has been computed are plane quintics (\cite{CL}) or cubic threefolds (\cite{DPT}).
    \item It is absolutely reasonable to ask what happens for those smooth projective varieties with $p_g=0$, for example for cubic hypersurfaces of dimension greater $3$ or for surfaces of general type with $p_g=0$.
    \item To conclude, it seems natural to investigate other {\em{maximal variation problems}}, related for example to the intersection of a smooth projective variety with linear spaces of codimension at least 2. 
\end{enumerate}

\bibliographystyle{amsalpha}

 \end{document}